\documentclass[pdftex]{amsart}	
\usepackage{amsmath, amsfonts, amssymb,amscd,graphics,graphicx,color,eucal,setspace}
\usepackage{latexsym}
\usepackage{color}
\usepackage[top=3.5cm, bottom=2.5cm, left=2.5cm, right=2.5cm, includefoot]{geometry}

\usepackage[all]{xy}
\usepackage{url}
\usepackage{amscd}
\usepackage{here}
\numberwithin{equation}{section}
\usepackage{bm}
\usepackage{tikz}

\theoremstyle{plain}
\newtheorem{theorem}{Theorem}[section]
\newtheorem{lemma}[theorem]{Lemma} 
\newtheorem{corollary}[theorem]{Corollary}

\theoremstyle{definition}
\newtheorem{remark}[theorem]{Remark}
\newtheorem{example}[theorem]{Example}
\newtheorem{definition}[theorem]{Definition}

\newcommand{\C}{\mathbb C}

\newcommand{\N}{\mathbb P}
\newcommand{\GX}{\mathcal{G}_X(h)}
\newcommand{\GY}{\mathcal{G}_Y(h)}
\newcommand{\G}{\mathcal{G}}
\newcommand{\Gc}{{}^\circ \G}
\newcommand{\tG}{\widetilde{\G}}
\newcommand{\tX}{\widetilde{X}}
\newcommand{\tY}{\widetilde{Y}}

\newcommand{\Fl}{\mathrm{Fl}}

\newcommand{\mclH}{\mathcal{H}}

\newcommand{\e}{\mathbf{e}}

\definecolor{gray7}{gray}{0.7}

\newcommand{\csf}[1][h]{\mathrm{csf}_{#1}(q)}
\newcommand{\Scsf}[1][h]{\mathrm{csf}_{#1}}
\newcommand{\LLT}[1][h]{\mathrm{LLT}_{#1}(q)}
\newcommand{\Sn}[1][n]{\mathfrak{S}_{#1}}
\newcommand{\asc}{\mathrm{asc}}
\newcommand{\Asc}{\mathrm{Asc}}
\newcommand{\TbU}{T \backslash \mathrm{U}(n)}

\newcommand{\wc}{{}^\circ w}
\newcommand{\vc}{{}^\circ v}
\newcommand{\Snc}{{}^\circ \Sn}

\def\Set#1{\Setdef#1\Setdef}
\def\Setdef#1|#2\Setdef{\left\{#1\;\mathstrut\vrule\;#2\right\}}%

\begin{document}

\title{Modular law through GKM theory}
\date{\today}

\author[T. Horiguchi]{Tatsuya Horiguchi}
\address{National institute of technology, Ube college, 2-14-1, Tokiwadai, Ube, Yamaguchi,
Japan 755-8555}
\email{tatsuya.horiguchi0103@gmail.com}

\author[M. Masuda]{Mikiya Masuda}
\address{Osaka Central Advanced Mathematical Institute, Sumiyoshi-ku, Osaka 558-8585, Japan.}
\email{mikiyamsd@gmail.com}

\author[T. Sato]{Takashi Sato}
\address{Osaka Central Advanced Mathematical Institute, Sumiyoshi-ku, Osaka 558-8585, Japan.}
\email{00tkshst00@gmail.com}


\date{\today}

\begin{abstract}
The solution of Shareshian-Wachs conjecture by Brosnan-Chow and Guay-Paquet tied the graded chromatic symmetric functions on indifference graphs (or unit interval graphs) and the cohomology of regular semisimple Hessenberg varieties with the dot action.
A similar result holds between unicellular LLT polynomials and twins of regular semisimple Hessenberg varieties.
A recent result by Abreu-Nigro enabled us to prove these results by showing the modular law for the geometrical objects, and this is indeed done by Precup-Sommers and Kiem-Lee.
In this paper, we give elementary and simpler proofs to the modular law through GKM theory. 
\end{abstract}

\maketitle

\setcounter{tocdepth}{1}

\section{Introduction}

Let $n$ be a positive integer and $[n]$ the set of integers from $1$ to $n$.
A function $h \colon [n]\to [n]$ is called a Hessenberg function if it is non-decreasing and $h(j)\ge j$ for all $j\in [n]$.
One may think of a Hessenberg function as a Dyck path. 

A graph $G_h$ (called an indifference graph) is associated to $h$ as follows:
\begin{enumerate}
	\item
	the vertex set $V(G_h)$ is $[n]$,
	\item
	the edge set $E(G_h)$ is $\{\{i,j\}\mid j<i\le h(j)\}$. 
\end{enumerate}
Let $\N$ be the set of positive integers.
A map $\kappa \colon [n]\to \N$ is called a $\N$-coloring on $G_h$ and it is proper if $\kappa(i)\not=\kappa(j)$ whenever $\{i,j\}\in E(G_h)$.
Let $z_1,z_2,\dots$ be infinitely many variables. We set ${\mathbf z}_\kappa:=z_{\kappa(1)}z_{\kappa(2)}\cdots z_{\kappa(n)}$ for a coloring $\kappa$ on $G_h$.
Then Stanley's chromatic symmetric function $\Scsf $ of $G_h$ is defined by 
\[
	\Scsf :=\sum_{\kappa\in PC(G_h)}{\mathbf z}_\kappa
\]
where $PC(G_h)$ denotes the set of all proper $\N$-colorings on $G_h$. 
The long-standing Stanley-Stembridge conjecture is known to be equivalent to $\Scsf$ being $e$-positive, that is, when $\Scsf$ is expressed as a polynomial in the elementary symmetric functions of $z_1,z_2,\dots$, all the coefficients are non-negative. 

Shareshian-Wachs \cite{sh-wa16} introduced a graded version of $\Scsf$ as follows: 
\begin{equation}
\label{eq:csf_q}
	\csf :=\sum_{\kappa\in PC(G_h)}{\mathbf z}_\kappa q^{\asc(\kappa)}
\end{equation}
where 
\[
	\asc(\kappa):=\#\{\{i,j\}\in E(G_h)\mid j<i,\ \kappa(j)<\kappa(i)\}.
\]
They show that the coefficients of $\csf$ as a polynomial in $q$ are symmetric functions in $z_1,z_2,\dots$.
A strong version of the Stanley-Stembridge conjecture is that those coefficients of $\csf$ are all $e$-positive. 

On the other hand, a regular semisimple Hessenberg variety $X(h)$ is associated with the Hessenberg function $h$.
It is a subvariety of the flag variety $\Fl(n)$ defined by 
\[
	X(h) := \{V_\bullet =(V_1\subset V_2\subset \cdots \subset V_n=\C^n)\in \Fl(n)\mid SV_i\subset V_{h(i)}\ \text{for $i\in [n]$}\},
\]
where $S$ is a linear endomorphism of $\C^n$ with distinct eigenvalues.
It is known that $X(h)$ is nonsingular and its smooth structure is independent of the choice of $S$. 
Since we are only concerned with its cohomology, we suppress $S$ in the notation of $X(h)$. 
The cohomology $H^*(X(h))$ of $X(h)$ concentrate on even degrees
and is a graded module over the symmetric group $\Sn$ on $[n]$ (see Section \ref{sect:3}). 

The following remarkable fact, which ties seemingly unrelated two objects above, was conjectured by Shareshian-Wachs \cite{sh-wa16} and proved by Brosnan-Chow \cite{br-ch18} and Guay-Paquet \cite{guay16}. 

\begin{theorem}[\cite{br-ch18}, \cite{guay16}]
\label{thm:1}
Let the situation be as above. Then 
\[
	\omega(\csf) = \sum_{i=0}^\infty \mathrm{ch}\left(H^{2i}(X(h))\right)q^i,
\] 
where $\omega$ denotes the involution on symmetric functions sending the $i$-th elementary symmetric function to the $i$-th complete symmetric function and $\mathrm{ch}$ denotes Frobenius characteristic sending $\Sn$-modules to symmetric functions of degree $n$. 
\end{theorem}

Recently, it was noticed in \cite{ma-sa23} and \cite{pr-so22} that a similar fact holds for unicellular LLT polynomials and the twin of $X(h)$.
LLT polynomials were originally introduced by Lascoux, Leclerc, and Thibon in \cite{LLT97}.
It can be seen as a $q$-deformation of the product of skew Schur functions and it is indexed by a tuple of skew Young diagrams.
An LLT polynomial is called \emph{unicellular} if each skew Young diagram in its index is a single box.
It is observed in \cite{ca-me17} that a unicellular LLT polynomial is associated with a Hessenberg function $h$ and can be expressed in terms of $\N$-coloring as 
\[
	\LLT = \sum_{\kappa\in C(G_h)}{\mathbf z}_\kappa q^{\asc(\kappa)}
\]
where $C(G_h)$ denotes the set of all $\N$-colorings on $G_h$ (the properness is not required). 

On the other hand, Ayzenberg-Buchstaber \cite{ay-bu21} introduced a closed smooth submanifold $Y(h)$ of $\Fl(n)$, which they call the twin of $X(h)$.
The twin $Y(h)$ resembles $X(h)$, e.g.\ $H^*(Y(h))$ is isomorphic to $H^*(X(h))$ as groups and $H^*(Y(h))$ is also a graded $\Sn$-module.
However, they are not isomorphic as rings and as $\Sn$-modules in general. 

\begin{theorem}[\cite{ma-sa23}, \cite{pr-so22}]
\label{thm:2}
Let the situation be as above. Then 
\[
	\LLT = \sum_{i=0}^\infty \mathrm{ch}\left(H^{2i}(Y(h))\right)q^i
\] 
where $\mathrm{ch}$ denotes Frobenius characteristic as before.
\end{theorem}

The proof of Theorem~\ref{thm:1} by Brosnan-Chow \cite{br-ch18} uses deep results in algebraic geometry and that by Guay-Paquet \cite{guay16} uses Hopf algebra on Dyck paths.
The proof of Theorem~\ref{thm:2} by Masuda-Sato \cite{ma-sa23} is based on Theorem~\ref{thm:1}.
Precup-Sommers \cite{pr-so22} prove both Theorems~\ref{thm:1} and \ref{thm:2} using intersection cohomology.
In fact, they prove that the right hand sides (i.e.\ the geometric sides) of the identities in Theorems~\ref{thm:1} and ~\ref{thm:2} satisfy the modular law.
Here a function $F$ on the set of Hessenberg functions taking values in a polynomial ring $\Lambda[q]$ over a ring $\Lambda$ is said to satisfy the modular law if
\[
	(1+q)F(h)=F(h_+)+qF(h_-)
\]
for every modular triple $(h_-,h,h_+)$ of Hessenberg functions, see Definition~\ref{dfn:modular_triple} for modular triples.
Abreu-Nigro \cite{ab-ni21} show that a function $F$ satisfying the modular law is uniquely determined by its ``initial conditions'', see Section 2 for details.
In Theorems~\ref{thm:1} and ~\ref{thm:2}, the left hand sides (i.e.\ the algebraic sides) of the identities are known to satisfy the modular law (see Appendices),
and the left-hand side and the right-hand side coincide for ``initial'' $h$'s
(see \cite[Proposition 2.14 and (2.17)]{ki-le22} and \cite[Theorem 1.1]{ab-ni21} for Theorem \ref{thm:1}, and see \cite[Lemma 3.11 and (3.20)]{ki-le23} and \cite[Theorem 2.4]{ab-ni21-2} for Theorem \ref{thm:2}).
Hence, once the modular law is established for the geometric sides, Theorems~\ref{thm:1} and ~\ref{thm:2} follow.

Recently, Kiem-Lee (\cite{ki-le22} and \cite{ki-le23}) also proved the modular law for the geometric sides.
Their proofs are elementary in the sense that their main tool is blow-up.
However they took different approaches between when proving Theorem \ref{thm:1} and when proving Theorem \ref{thm:2}.
Abreu-Nigro \cite[Example 3.5]{ab-ni22} also proved Theorem \ref{thm:1}, but they omitted the details.

In this paper, we prove the modular law for the geometric sides through GKM theory.
Our proof is motivated by the blow-up idea of Kiem-Lee and may be regarded as graph analogue of their proofs.
However ours is more elementary and simpler.
Moreover, our proofs for $X(h)$ and $Y(h)$ proceed in the same way.
Indeed, we blow-up the GKM graph of $X(h_+)$ (or $Y(h_+)$) along the GKM graph of $X(h_-)$ (or $Y(h_-)$).
The resulting graph, which corresponds to a roof manifold in \cite{ki-le23}, is a labeled graph but not a GKM graph.
We consider its graph cohomology satisfying a certain condition and compute it in two ways.
The modular law for the geometric sides follows by comparing the two expressions of the graph cohomology.
We also provide a simple elementary proof to the modular law for the algebraic sides for the reader's convenience in the appendix. 

This paper is organized as follows. 
In Section~\ref{sect:2} we explain a modular triple. We also recall the uniqueness result of Abreu-Nigro \cite{ab-ni21}.
In Section~\ref{sect:3} we briefly review GKM theory which is our main tool.
In Section~\ref{sect:4} we set up notations used for our proof of the modular law for the geometric sides.
We give the proof of the modular law for $X(h)$ in Section~\ref{sect:5}.
In Section \ref{sect:6} we point out the necessary change for the proof of the modular law for $Y(h)$.
All cohomology groups are taken with $\C$ coefficients throughout this paper unless otherwise stated. 

\section{Modular triple and modular law}
\label{sect:2}

We denote the set of all Hessenberg functions on $[n]$ by $\mclH(n)$. We often express $h\in \mclH(n)$ as a vector $(h(1),\dots,h(n))$ by listing its values.
It is also convenient to visualize $h$ by drawing a configuration of the shaded boxes on a square grid of size $n\times n$, which consists of boxes in the $i$-th row and the $j$-th column satisfying $i\le h(j)$.
Since $h(j)\ge j$ for any $j\in [n]$, the essential part is the shaded boxes below the diagonal, see Figure~\ref{fig:1} below. 

\begin{figure}[H]
\begin{center}
\scalebox{0.7}
{
\setlength{\unitlength}{5mm}
\begin{picture}(25,11)(0,0)
	\multiput(0,10.2)(0,-1){2}{\colorbox{gray7}{\phantom{\vrule width 3mm height 3mm}}}
	\multiput(1,10.2)(0,-1){5}{\colorbox{gray7}{\phantom{\vrule width 3mm height 3mm}}}
	\multiput(2,10.2)(0,-1){6}{\colorbox{gray7}{\phantom{\vrule width 3mm height 3mm}}}
	\multiput(3,10.2)(0,-1){8}{\colorbox{gray7}{\phantom{\vrule width 3mm height 3mm}}}
	\multiput(4,10.2)(0,-1){9}{\colorbox{gray7}{\phantom{\vrule width 3mm height 3mm}}}
	\multiput(5,10.2)(0,-1){9}{\colorbox{gray7}{\phantom{\vrule width 3mm height 3mm}}}
	\multiput(6,10.2)(0,-1){11}{\colorbox{gray7}{\phantom{\vrule width 3mm height 3mm}}}
	\multiput(7,10.2)(0,-1){11}{\colorbox{gray7}{\phantom{\vrule width 3mm height 3mm}}}
	\multiput(8,10.2)(0,-1){11}{\colorbox{gray7}{\phantom{\vrule width 3mm height 3mm}}}
	\multiput(9,10.2)(0,-1){11}{\colorbox{gray7}{\phantom{\vrule width 3mm height 3mm}}}
	\multiput(10,10.2)(0,-1){11}{\colorbox{gray7}{\phantom{\vrule width 3mm height 3mm}}}
	\linethickness{0.3mm}
	\multiput(1,0)(1,0){10}{\line(0,1){11}}
	\multiput(0,1)(0,1){10}{\line(1,0){11}}
	\multiput(0,11)(1,-1){11}{\line(1,-1){1}}
	\put(0,0){\framebox(11,11)}
	\multiput(14,10.2)(0,-1){5}{\colorbox{gray7}{\phantom{\vrule width 3mm height 3mm}}}
	\multiput(15,10.2)(0,-1){5}{\colorbox{gray7}{\phantom{\vrule width 3mm height 3mm}}}
	\multiput(16,10.2)(0,-1){7}{\colorbox{gray7}{\phantom{\vrule width 3mm height 3mm}}}
	\multiput(17,10.2)(0,-1){8}{\colorbox{gray7}{\phantom{\vrule width 3mm height 3mm}}}
	\multiput(18,10.2)(0,-1){8}{\colorbox{gray7}{\phantom{\vrule width 3mm height 3mm}}}
	\multiput(19,10.2)(0,-1){9}{\colorbox{gray7}{\phantom{\vrule width 3mm height 3mm}}}
	\multiput(20,10.2)(0,-1){10}{\colorbox{gray7}{\phantom{\vrule width 3mm height 3mm}}}
	\multiput(21,10.2)(0,-1){10}{\colorbox{gray7}{\phantom{\vrule width 3mm height 3mm}}}
	\multiput(22,10.2)(0,-1){10}{\colorbox{gray7}{\phantom{\vrule width 3mm height 3mm}}}
	\multiput(23,10.2)(0,-1){11}{\colorbox{gray7}{\phantom{\vrule width 3mm height 3mm}}}
	\multiput(24,10.2)(0,-1){11}{\colorbox{gray7}{\phantom{\vrule width 3mm height 3mm}}}
	\linethickness{0.3mm}
	\multiput(15,0)(1,0){10}{\line(0,1){11}}
	\multiput(14,1)(0,1){10}{\line(1,0){11}}
	\multiput(14,11)(1,-1){11}{\line(1,-1){1}}
	\put(14,0){\framebox(11,11)}
\end{picture}
}
\end{center}
\caption{The configurations for $h=(2,5,6,8,9,9,11,11,11,11,11)$ and $(5,5,7,8,8,9,10,10,10,11,11)$.}
\label{fig:1}
\end{figure}
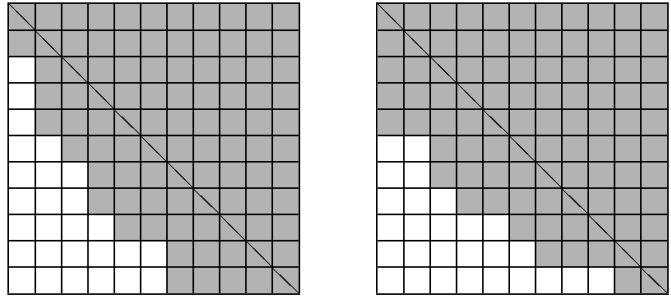

If we flip the configuration of $h$ along the anti-diagonal, then the resulting one is again a configuration of a Hessenberg function, denoted by $h^t$. We call $h^t$ the transpose of $h$. 
For example, the two Hessenberg functions in Figure~\ref{fig:1} are the transposes of each other. 

We introduce the following terminology used in \cite{ki-le23}. 
\begin{definition}[Modular triple]
\label{dfn:modular_triple}
Let $h_-,h,h_+$ be elements in $\mclH(n)$. The triple $(h_-,h,h_+)$ is called a \emph{modular triple} if it satisfies one of the following. 
\begin{enumerate}
\item[(C)] If $h(d)=h(d+1)$ and $h^{-1}(d)=\{d_0\}$ for some $1\le d_0<d<n$, then $h_-$ and $h_+$ are defined by 
\[
	h_-(j) =
	\begin{cases}
		d-1\quad &\text{for $j=d_0$}\\
		h(j) \quad&\text{otherwise}
	\end{cases}
	\qquad\text{and}\qquad 
	h_+(j)=
	\begin{cases}
		d+1\quad &\text{for $j=d_0$}\\
		h(j) \quad&\text{otherwise}.
	\end{cases}
\]
\item[(R)] If $h(d')+1=h(d'+1)\not=d'+1$ and $h^{-1}(d')=\emptyset$ for some $1\le d'<n$, then $h_-$ and $h_+$ are defined by 
\[
	h_-(j)=
	\begin{cases}
		h(d')\quad &\text{for $j=d'+1$}\\
		h(j) \quad&\text{otherwise}
	\end{cases}
	\qquad\text{and}\qquad 
	h_+(j)=
	\begin{cases}
		h(d')+1\quad &\text{for $j=d'$}\\
		h(j) \quad&\text{otherwise}.
	\end{cases}
\]
\end{enumerate}
\end{definition}

\begin{remark}
\label{rem:1}
In Definition \ref{dfn:modular_triple}, (C) stands for column and (R) stands for row.
The statements (C) and (R) seem unrelated but they are related through $h\to h^t$, e.g.\ (C) for $h^t$ implies (R) for $h$. 
\end{remark}

\begin{example}
For the left $h$ in Figure~\ref{fig:1}, $(d,d_0)=(5,2), (8,4)$ satisfy (C) while $d'=4$ satisfies (R).
For the right $h$ in Figure~\ref{fig:1}, $(d,d_0)=(7,3)$ satisfies (C) while $d'= 3, 6$ satisfy (R).
Since they are the transposes to each other,
the $(d+1)$-th column of the left $h$ corresponds to the $d'$-th row of the right $h$ by $n+1-d' = d+1$,
that is, $d + d' = n$, where $n = 11$ in this case.
In other words, $d' = w_0(d+1)$, where $w_0$ is the longest element of $\Sn$.
\end{example}

\begin{definition}
Let $\Lambda[q]$ be a polynomial ring over a ring $\Lambda$ (in fact, we take $\Lambda$ to be 
the ring of symmetric functions in $z_1,z_2,\dots$).
We say that a function $F\colon \mclH(n)\to \Lambda[q]$ satisfies the \emph{modular law} if 
\begin{equation}
\label{eq:modular_relation}
	(1+q)F(h)=F(h_+)+qF(h_-)
\end{equation}
for every modular triple $(h_-,h,h_+)$. 
\end{definition}

\begin{remark}
\label{rem:2}
If $F(h^t)=F(h)$ for any $h\in \mclH(n)$, then it suffices to check the modular relation \eqref{eq:modular_relation} for modular triples of type (C) for $F$ by Remark~\ref{rem:1}. 
\end{remark}

Given $h_1\in \mclH(n_1)$ and $h_2\in \mclH(n_2)$, their product $h_1h_2\in \mclH(n_1+n_2)$ can naturally be defined by 
\[
	h_1h_2(j) =
	\begin{cases}
		h_1(j)\quad&\text{for $1\le j\le n_1$}\\
		h_2(j-n_1)+n_1\quad&\text{for $n_1+1\le j\le n_1+n_2$}.
	\end{cases}
\] 

\begin{theorem}[Abreu-Nigro \cite{ab-ni21}]
\label{thm:ab-ni}
A function $F\colon \mclH(n)\to \Lambda[q]$ which satisfies the modular law is determined by it values at products $h_1h_2\cdots h_r$ of all tuples $(h_1,\dots,h_r)$ such that $h_i\in \mclH(n_i)$, 
$h_i(j)=n_i$ for $\forall j\in [n_i]$, and $\sum_{i=1}^r n_i=n$. 
\end{theorem}

%

\section{Graph cohomology of a labeled graph and GKM theory}
\label{sect:3}

Let $T$ be the compact torus $(S^1)^n$ where $S^1$ is the unit circle of $\C$.
The classifying space $BT$ of $T$ is $(\C P^\infty)^n$.
We choose a generator of $H^2(\C P^\infty)$ obtained as the first Chern class of the tautological line bundle over $\C P^\infty$ and let $t_1,\dots,t_n$ be a generator of $H^2(BT)$ coming from the factors of $BT=(\C P^\infty)^n$.
Then $H^*(BT)$ is a polynomial ring in $t_1,\dots,t_n$. 

Let $\Gamma=(V,E,\alpha)$ be a \emph{labeled graph}, where $V$ is a vertex set, $E$ is an edge set, and $\alpha\colon E\to H^2(BT)$ is a labeling on $E$. 
For $e\in E$, we denote by $e_{\pm}$ the endpoints of $e$.
Then we define the (equivariant) \emph{graph cohomology} of the labeled graph $\Gamma$ by 
\begin{equation}
\label{eq:graph_cohomology}
	H^*_T(\Gamma)
	:= \{f\in \mathrm{Map}(V,H^*(BT))\mid f(e_+)\equiv f(e_-)\pmod{\alpha(e)}
		\quad\text{for any $e\in E$}\}.
\end{equation}
Any constant map on $V$ taking a value $t$ in $H^*(BT)$ is an element of $H^*_T(\Gamma)$, which we also denote by $t$, and $H^*_T(\Gamma)$ is a module over $H^*(BT)$ in a natural way. We define 
\begin{equation}
\label{eq:H*graph} 
	H^*(\Gamma):=H^*_T(\Gamma)/(t_1,\dots,t_n)
\end{equation}
where $(\ )$ denotes the ideal generated by the elements in it. 

Graph cohomology often arises as the equivariant cohomology 
\[
	H^*_T(X):=H^*(ET\times_T X)
\]
of a nice $T$-space $X$ called a GKM manifold (see \cite{gu-za01}), where $ET\to BT$ is the universal principal $T$-bundle and $ET\times_T X$ is the orbit space of $ET\times X$ by the diagonal $T$-action.
The first projection $ET\times X\to ET$ induces a fibration 
\[
	X \to ET\times_T X \xrightarrow{\pi} ET/T=BT,
\]
so $H^*_T(X)$ is a module over $H^*(BT)$ through $\pi^*\colon H^*(BT)\to H^*_T(X)$.
If $H^*(X)$ concentrate on even degrees, then the Serre spectral sequence of the fibration above collapses;
so the restriction map $H^*_T(X)\to H^*(X)$ is surjective and its kernel is the ideal generated by $\pi^*(H^{>0}(BT))$, so that we obtain an isomorphism 
\begin{equation}
\label{eq:H*X}
	H^*_T(X)/(\pi^*(t_1),\dots,\pi^*(t_n))\cong H^*(X).
\end{equation}

We say that the labeled graph $\Gamma$ is \emph{$2$-independent} if, for any $p\in V$ and edges $e_1,\dots,e_m$ incident to $p$, $\alpha(e_1),\dots,\alpha(e_m)$ are pairwise linearly independent.
The equivariant cohomology ring of a GKM manifold $X$ is recovered from its fixed point set and $1$-dimensional orbits, and they form a $2$-independent labeled graph called a GKM graph (see \cite{gu-za01} for details).

Recall that
\begin{equation}
\label{eq:Xh}
	X(h)=\{V_\bullet =(V_1\subset V_2\subset \cdots
	\subset V_n=\C^n)\in \Fl(n)\mid SV_i\subset V_{h(i)}\ \text{for $i\in [n]$}\}
\end{equation}
where $S$ is a linear endomorphism of $\C^n$ with distinct eigenvalues.
As remarked in the introduction, the diffeomorphism type of $X(h)$ is independent of the choice of $S$.
We take $S$ to be a linear operator defined by a diagonal matrix with distinct eigenvalues.
For later use, we assume that they are real numbers.
Then $S$ commutes with the standard action of $T=(S^1)^n$ on $\C^n$ defined by coordinate-wise multiplication,
so the induced $T$-action on $\Fl(n)$ leaves $X(h)$ invariant.
One can easily check that the $T$-fixed point sets $X(h)^T$ and $\Fl(n)^T$ consist of permutation flags $V_\bullet(w)$ associated with elements $w\in \Sn$;
\[
	V_\bullet(w) = (\langle \e_{w(1)}\rangle \subset \langle \e_{w(1)},\e_{w(2)}\rangle \subset \cdots\subset \langle\e_{w(1)},\dots,\e_{w(n)}\rangle=\C^n),
\]
where $\e_1,\dots,\e_n$ denotes the standard basis of $\C^n$ and $\langle \ \rangle$ denote the linear subspace of $\C^n$ spanned by the elements in it.
In the following, we make the following identification 
\[
	X(h)^T=\Fl(n)^T=\Sn.
\]

The GKM graph $\GX=(\Sn,E(h),\alpha_X)$ associated to $X(h)$ with the $T$-action is given by 
\[
	E(h)=\{\{w,w(i,j)\}\mid w\in \Sn,\ j<i\le h(j)\}
\]
and
\begin{equation}
\label{eq:alpha_X}
	\alpha_X(\{w,w(i,j)\})=t_{w(i)}-t_{w(j)},
\end{equation}
where $(i,j)$ denotes the transposition exchanging $i$ and $j$, see \cite{tymo08}.
Then $H^*_T(\GX)$ consists of all $f\in \mathrm{Map}(V,H^*(BT))$ satisfying
\begin{equation*}
	f(w) - f(w(i,j)) \equiv 0 \mod (t_{w(i)} - t_{w(j)})
\end{equation*}
for any $i < j \leq h(i)$ and $w \in \Sn$.
Note that any edge $\{w,w(i,j)\}$ of $\GX$ corresponds to $\C P^1$ in $X(h)$ which contains $V_\bullet(w)$ and $V_\bullet(w(i,j))$, see \cite[Subsection 2.2]{ki-le22}.
It is known that $H^*(X(h))$ concentrate on even degrees (see \cite{ma-pr-sh92}), so the restriction map 
\[
	\iota^*\colon H^*_T(X(h))\to H^*_T(X(h)^T)
	=\bigoplus_{w\in \Sn}H^*_T(w)=\mathrm{Map}(\Sn,H^*(BT))
\]
is injective.
GKM theory (\cite{go-ko-ma98}) tells us that the image of $\iota^*$ is $H^*_T(\GX)$, so we have an isomorphism 
\begin{equation}
\label{eq:equiv_iso}
	\iota^*\colon H^*_T(X(h))\xrightarrow{\cong} H^*_T(\GX).
\end{equation}
Through $\iota^*$, $\pi^*(t_i)$ in \eqref{eq:H*X} corresponds to the constant map $t_i$ on $\Sn$.
Therefore, it follows from \eqref{eq:H*graph} and \eqref{eq:H*X} that the isomorphism \eqref{eq:equiv_iso} reduces to an isomorphism 
\begin{equation}
\label{eq:iso}
	H^*(X(h))\xrightarrow{\cong} H^*(\GX). 
\end{equation}

We consider an action of $\sigma\in \Sn$ on $H^*(BT)$ induced by sending $t_i$ to $t_{\sigma(i)}$ for $i\in [n]$,
and define an action of $\sigma\in \Sn$ on $\mathrm{Map}(\Sn,H^*(BT))$ by 
\begin{equation}
\label{eq:dot_action}
	(\sigma\cdot f)(w):=\sigma(f(\sigma^{-1}w))\qquad
	\text{for $f\in \mathrm{Map}(\Sn,H^*(BT))$ and $w\in \Sn$}.
\end{equation} 
This action was considered by Tymoczko \cite{tymo08}.
It preserves not only $H^*_T(\GX)$ but also the ideal $(t_1,\dots,t_n)$ in $H^*_T(\GX)$,
so the action descends to an action of $\Sn$ on $H^*(\GX)$.
Thus we obtain actions of $\Sn$ on $H^*_T(X(h))$ and $H^*(X(h))$ through the isomorphisms \eqref{eq:equiv_iso} and \eqref{eq:iso}.
These actions are called the \emph{dot action}. 

A similar story holds for the twin $Y(h)$ of $X(h)$ introduced by Ayzenberg-Buchstaber \cite{ay-bu21}.
The twin $Y(h)$ is defined as follows.
We regard $\Fl(n)$ as the homogeneous space $\mathrm{U}(n)/T$ where $\mathrm{U}(n)$ is the unitary group of size $n$ and $T$ is the torus consisting of diagonal matrices in $\mathrm{U}(n)$.
Indeed, a unitary matrix $g=[v_1,\dots,v_n]\in \mathrm{U}(n)$ associates a flag 
\begin{equation}
\label{eq:g-flag}
	\langle v_1\rangle \subset \langle v_1,v_2\rangle \subset \cdots \subset \langle v_1,\dots,v_n\rangle=\C^n
\end{equation}
and this correspondence induces the identification $\mathrm{U}(n)/T=\Fl(n)$.
In particular, the permutation matrix $P(w)$ corresponds to the permutation flag $V_\bullet(w)$, where the $(w(k),k)$-components of $P(w)$ are $1$ for any $k$ and the others components are $0$.
If the flag $V_\bullet$ in \eqref{eq:Xh} is of the form \eqref{eq:g-flag},
then the condition $SV_i\subset V_{h(i)}$ for $i\in [n]$ in \eqref{eq:Xh} can be written as 
\[
	Sg\in gH\qquad \text{i.e.}\qquad g^{-1}Sg\in H,
\]
where $H$ is the vector subspace consisting of $n\times n$ matrices $(a_{ij})$ with $a_{ij}=0$ for $i>h(j)$.
Therefore 
\begin{equation}
\label{eq:Xh_in_Un/T}
	X(h) = \{gT\in \mathrm{U}(n)/T\mid g^{-1}Sg\in H\}.
\end{equation}
The twin $Y(h)$ of $X(h)$ is simply defined as 
\begin{equation}
\label{eq:Yh}
	Y(h) := \{Tg\in \TbU \mid g^{-1}Sg\in H\}. 
\end{equation}
It is a closed smooth manifold but not necessarily an algebraic variety.
The cohomology $H^*(Y(h))$ is isomorphic to $H^*(X(h))$ as groups.
In particular $H^*(Y(h))$ concentrate on even degrees.
However, they are not isomorphic as rings in general.
See \cite[Theorem 3.10 and Remark 3.11]{ay-bu21}.
Note that Ayzenberg-Buchstaber write the twin as $X_h$ and the regular semisimple Hessenberg variety as $Y_h$.

\begin{remark}
	By taking inverse matrices, $Y(h)$ can be defined as
	\[
		Y(h)=\{g/T\in \mathrm{U}(n)/T\mid gSg^{-1}\in H\},
	\]
	and the second and the third authors adopted this definition in \cite{ma-sa23}.
\end{remark}

The right multiplication by $T$ on $\mathrm{U}(n)$ induces the action of $T$ on $\TbU$ which leaves $Y(h)$ invariant.
One can check that 
\[
	Y(h)^T = (\TbU)^T= \{T P(w) \mid w \in \Sn \} =\Sn.
\]
According to \cite[Proposition 5.3]{ay-bu21}, $Y(h)$ with the $T$-action is a GKM manifold and its GKM graph $\GY$ is $(\Sn,E(h),\alpha_Y)$,
where the labeling $\alpha_Y\colon E(h)\to H^2(BT)$ is given by
\begin{equation}\label{eq:alpha_Y}
	\alpha_Y(\{w,w(i,j)\})=t_i-t_j.
\end{equation}
In particular, the underlying graphs of $\GX$ and $\GY$ are the same but the labelings $\alpha_X$ and $\alpha_Y$ are different.
Compare \eqref{eq:alpha_X} and \eqref{eq:alpha_Y}. 
The isomorphisms \eqref{eq:equiv_iso} and \eqref{eq:iso} hold for $Y(h)$ and $\GY$.
However, the dot action \eqref{eq:dot_action} on $\mathrm{Map}(\Sn,H^*(BT))$ does not preserve $H^*_T(\GY)$.
As for $Y(h)$ and $\GY$,
we consider the action of $\sigma\in \Sn$ on $\mathrm{Map}(\Sn,H^*(BT))$ defined by 
\begin{equation}
\label{eq:dagger_action}
	(\sigma\dagger f)(w) := f(\sigma^{-1}w) \qquad
	\text{for $f\in \mathrm{Map}(\Sn,H^*(BT))$ and $w\in \Sn$}. 
\end{equation} 
This action preserves $H^*_T(\GY)$ and induces actions of $\Sn$ on $H^*_T(Y(h))$ and $H^*(Y(h))$, called the dagger action in \cite{ma-sa23}. 

\begin{remark}
\label{rem:Xht}
	Let $J$ be the anti-diagonal matrix with all entry $1$.
	Note that the anti-diagonal transpose of a matrix $A$ is given by $J({}^t\! A) J$,
	where ${}^t\! A$ is the ordinary transpose of $A$.
	In the sense of \eqref{eq:Xh_in_Un/T},
	the correspondence $gT \mapsto gJ T$ gives a diffeomorphism $X(h) \cong X(h^t)$,
	since $g^{-1} = {}^t \bar{g}$, $S = {}^t \!S$, $\bar{S} = S$, and $\bar{H} = H$.
	This correspondence means the orthogonal complement of flags;
	for orthogonal $v_1,\ldots, v_n$,
	\[
		(\{0\} \subset \langle v_1 \rangle \subset \langle v_1, v_2 \rangle \subset \cdots \langle v_1,\ldots, v_{n-1} \rangle \subset \C^n) \mapsto (\C^n \supset \langle v_2,\ldots, v_n \rangle \supset \langle v_3,\ldots, v_n \rangle \supset \cdots \langle v_n \rangle \supset \{0\}).
	\]
	In terms of GKM graphs, an edge $\{w, w(i,j)\}$ corresponds to $\{ww_0, w(i,j)w_0\}$
	and they have the same label.
	Then this gives an isomorphism between GKM graphs of $X(h)$ and $X(h^t)$.
	In particular, this induces an $H^*(BT)$-algebra isomorphism $H^*_T(X(h^t)) \to H^*_T(X(h))$ which commutes with the dot action.
\end{remark}

\section{Setting for the proof of the modular law}
\label{sect:4}

In this section we set up notations for the proof of the modular law for $X(h)$.
The same argument works for $Y(h)$ with a little modification and we point out the modification in Section \ref{sect:6}. 

Since $H^*_T(X(h))$ is isomorphic to $H^*_T(X(h^t))$ as graded $\Sn$-modules (see Remark \ref{rem:Xht}),
it suffices to check the modular relation \eqref{eq:modular_relation} for modular triples $(h_-,h,h_+)$ of type (C) in Definition~\ref{dfn:modular_triple} (see Remark~\ref{rem:2}); so 
\[
h(d)=h(d+1)\quad\text{and}\quad h^{-1}(d)=\{d_0\}\quad \text{for some $1\le d_0<d<n$}
\]
and 
\[
h_-(j)=\begin{cases} d-1\quad &\text{for $j=d_0$}\\
h(j) \quad&\text{otherwise}\end{cases}\qquad\text{and}\qquad 
h_+(j)=\begin{cases} d+1\quad &\text{for $j=d_0$}\\
h(j) \quad&\text{otherwise}.\end{cases}
\]
To simplify the notations, we set
\[
\G_-:=\G_X(h_-),\quad \G:=\G_X(h),\quad \G_+:=\G_X(h_+)
\]
where the vertex sets are all $\Sn$ and the edge sets are respectively
\begin{itemize}
	\item
	$E(\G_-) = \{\{w,w(i,j)\}\mid w\in \Sn,\ j<i\le h_-(j)\} = E(\G)\backslash\{\{w,w(d,d_0)\}\mid w\in\Sn\}$,
	\item
	$E(\G) =\{\{w,w(i,j)\}\mid w\in \Sn,\ j<i\le h(j)\}$,
	\item
	$E(\G_+) = \{\{w,w(i,j)\}\mid w\in \Sn,\ j<i\le h_+(j)\} =E(\G)\cup \{\{w,w(d+1,d_0)\}\mid w\in\Sn\}$,
\end{itemize}
\noindent
and the labelings on them are the same as $\alpha_X$ in \eqref{eq:alpha_X}, see Figure \ref{fig:MT}.

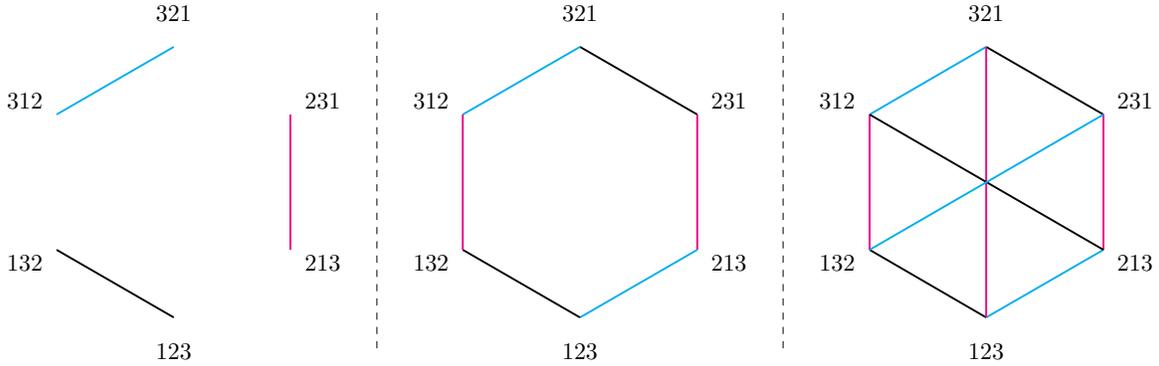
\begin{figure}[H]
\begin{center}
\scalebox{0.9}{
\begin{tikzpicture}
	\draw [thick, cyan] (-7.732,1)--(-6,2);
	\draw [thick, magenta] (-4.278,1)--(-4.278,-1);
	\draw [thick] (-6,-2)-- (-7.732,-1);
	\draw(-6,-2.5) node{$123$};
	\draw(-6,2.5) node{$321$};
	\draw(-8.2,-1.2) node{$132$};
	\draw(-8.2,1.2) node{$312$};
	\draw(-3.8,-1.2) node{$213$};
	\draw(-3.8,1.2) node{$231$};
	\draw [dashed] (-3,2.5)--(-3,-2.5);
	\draw [thick] (0,2)--(1.732,1);
	\draw [thick, magenta] (1.732,1)--(1.732,-1);
	\draw [thick, cyan] (1.732,-1)--(0,-2);
	\draw [thick] (0,-2)--(-1.732,-1);
	\draw [thick, magenta] (-1.732,-1)--(-1.732,1);
	\draw [thick, cyan] (-1.732,1)--(0,2);
	\draw(0,-2.5) node{$123$};
	\draw(0,2.5) node{$321$};
	\draw(-2.2,-1.2) node{$132$};
	\draw(-2.2,1.2) node{$312$};
	\draw(2.2,-1.2) node{$213$};
	\draw(2.2,1.2) node{$231$};
	\draw [dashed] (3,2.5)--(3,-2.5);
	\draw [thick] (6,2)--(7.732,1);
	\draw [thick, magenta] (7.732,1)--(7.732,-1);
	\draw [thick, cyan] (7.732,-1)--(6,-2);
	\draw [thick] (6,-2)--(4.278,-1);
	\draw [thick, magenta] (4.278,-1)--(4.278,1);
	\draw [thick, cyan] (4.278,1)--(6,2);
	\draw [thick, magenta] (6,2)--(6,-2);
	\draw [thick] (4.278,1)--(7.732,-1);
	\draw [thick, cyan] (4.278,-1)--(7.732,1);
	\draw(6,-2.5) node{$123$};
	\draw(6,2.5) node{$321$};
	\draw(3.8,-1.2) node{$132$};
	\draw(3.8,1.2) node{$312$};
	\draw(8.2,-1.2) node{$213$};
	\draw(8.2,1.2) node{$231$};
	\draw (0,-3.3) node{The label on black edges is $t_d - t_{d+1} = t_2-t_3$, that on \textcolor{cyan}{cyan edges} is $t_1-t_2$, and that on \textcolor{magenta}{magenta edges} is $t_1 -t_3$.};
\end{tikzpicture}
}
\caption{GKM graphs $\G_-$, $\G$, and $\G_+$ when $h = (2,3,3)$ and $d = 2$.}
\label{fig:MT} 
\end{center}
\end{figure}

We consider two more labeled graphs $\Gc$ and $\tG$ defined as follows: 
\begin{enumerate}
\item $V(\Gc)=\{\wc\mid w\in \Sn\}$ where $\wc$ is a copy of $w$, 
\item $E(\Gc)= \{\{\wc,\wc(i,j)\}\mid j<i\le h_-(j)\}\cup \{\{\wc,\wc(d+1,d_0)\}\}$, 
where $w\in \Sn$ and
 the label on $\{\wc,\wc(i,j)\}$ is $t_{w(i)}-t_{w(j)}$,
\end{enumerate}
and 
\begin{enumerate}
\item $V(\tG)=V(\G)\cup V(\Gc)$,
\item $E(\tG)=E(\G)\cup E(\Gc)\cup \{\{w,\wc\}\mid w\in \Sn\}$,
where the label on the edge $\{w,\wc\}$ is $t_{w(d+1)}-t_{w(d)}$,
\end{enumerate}
see Figure \ref{fig:tG}.

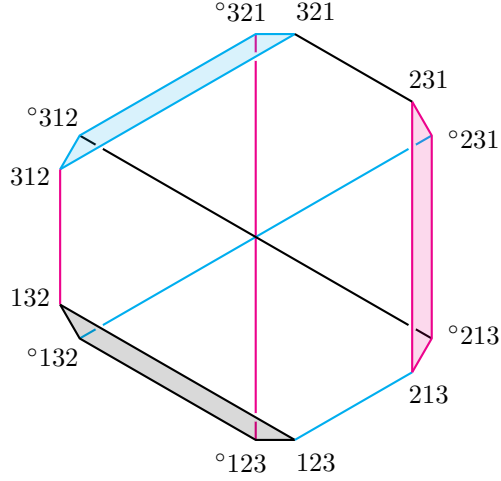
\begin{figure}[H]
\begin{center}
\begin{tikzpicture}
	\draw [thick, magenta] (0,2.7)--(0,-2.7);
	\draw [thick, cyan] (2.34,1.35)--(-2.34,-1.35);
	\draw [thick] (2.34,-1.35)--(-2.34,1.35);
	\filldraw [white](0,2.4) circle (0.05);
	\filldraw [white](0,-2.4) circle (0.05);
	\filldraw [white](-2.08,1.2) circle (0.05);
	\filldraw [white](-2.08,-1.2) circle (0.05);
	\filldraw [white](2.08,1.2) circle (0.05);
	\filldraw [white](2.08,-1.2) circle (0.05);
	\draw [fill = cyan, opacity=0.15] (-2.6,0.9)--(-2.34,1.35)--(0,2.7)--(0.52,2.7);
	\draw [thick, cyan] (-2.6,0.9)--(-2.34,1.35)--(0,2.7)--(0.52,2.7);
	\draw [thick, cyan] (0.52,2.7)--(-2.6,0.9);
	\draw [thick, magenta] (-2.6,0.9)--(-2.6,-0.9);
	\draw [fill = black, opacity=0.15] (-2.6,-0.9)--(-2.34,-1.35)--(0,-2.7)--(0.52,-2.7);
	\draw [thick] (-2.6,-0.9)--(-2.34,-1.35)--(0,-2.7)--(0.52,-2.7);
	\draw [thick] (0.52,-2.7)--(-2.6,-0.9);
	\draw [thick, cyan] (0.52,-2.7)--(2.08,-1.8);
	\draw [fill = magenta, opacity=0.15] (2.08,1.8)--(2.34,1.35)--(2.34,-1.35)--(2.08,-1.8);
	\draw [thick, magenta] (2.08,1.8)--(2.34,1.35)--(2.34,-1.35)--(2.08,-1.8);
	\draw [thick, magenta] (2.08,-1.8)--(2.08,1.8);
	\draw [thick] (2.08,1.8)--(0.52,2.7);
	\draw(0.8,-3) node{$123$};
	\draw(-0.2,-3) node{${}^\circ123$};
	\draw(0.8,3) node{$321$};
	\draw(-0.2,3) node{${}^\circ321$};
	\draw(2.3,-2.1) node{$213$};
	\draw(2.9,-1.3) node{${}^\circ213$};
	\draw(2.3,2.1) node{$231$};
	\draw(2.9,1.3) node{${}^\circ231$};
	\draw(-3,-0.8) node{$132$};
	\draw(-2.7,-1.6) node{${}^\circ132$};
	\draw(-3,0.8) node{$312$};
	\draw(-2.7,1.6) node{${}^\circ312$};
\end{tikzpicture}
\caption{Labeled graph $\tG$ with emphasized $4$-gons when $h = (2,3,3)$ and $d = 2$.}
\label{fig:tG}
\end{center}
\end{figure}

As a labeled graph,
$\tG$ is considered as blowing up $\G_+$ along the subgraph $\G_-$ (cf.\ \cite[Example 7]{gu-za00}).
Note that $\G, \Gc$, and $\G_-$ are full subgraphs of $\tG$.
We denote
\[
	\tau = (d+1, d) \in \Sn
\]
in the following.

\begin{lemma}
\label{lem:1}
The map $\Phi\colon \G\to \Gc$ sending $w\in V(\G)$ to $\wc\tau\in V(\Gc)$ gives a bijection between their edges preserving the labels, so $\G$ and $\Gc$ are isomorphic as labeled graphs. 
\end{lemma}

\begin{proof}
Noting $(i,j)\tau=\tau(\tau(i),\tau(j))$, 
one can easily check that $\{w,w(i,j)\}$ is an edge of $\G$ if and only if $\{\wc\tau,\wc(i,j)\tau\}$ is an edge of $\Gc$ and that the label on the edge $\{\wc\tau,\wc(i,j)\tau\}$ of $\Gc$ is $t_{w(i)}-t_{w(j)}$ which agrees with the label on $\{w,w(i,j)\}$ of $\G$. This proves the lemma. 
\end{proof}

Unfortunately, $\tG$ is not a GKM graph of any GKM manifold 
since $\tG$ is not $2$-independent.
 Indeed, the four edges 
\[
	\{w,w\tau\}, \ \{\wc, \wc\tau\},\ \{w,\wc\},\ \{w\tau,\wc\tau\},
\] 
which form a $4$-gon, have the same label $t_{w(d+1)}-t_{w(d)}$ for each $w\in \Sn$, see Figure \ref{fig:tG}.
We shall assign $\tG$ an additional data. 
It is a map $s \colon V(\tG) \to \{\pm1\}$ defined by
\[
	s(w) = +1,\quad s({}^\circ w) = -1 \quad \text{for any } w \in \Sn.
\]
The geometrical meaning of this sign is explained in Appendix \ref{app:tX}.
We define the equivariant cohomology ring of the pair
$(\tG, s)$ as follows.
Let $K_w$ be the $4$-gon with vertex set $\{w, {}^\circ w, w\tau, {}^\circ w \tau \}$.
Then
\[
	H^*_T(\tG, s) = \Set{f \in H^*_T(\tG) | \sum_{v \in V(K_w)} s(v)f(v) \equiv 0 \bmod{(t_{w(d)} - t_{w(d+1)})^2} \text{ for any } w \in \Sn}.
\]
In other words, $H^*_T(\tG, s)$ consists of all $f \in H^*_T(\tG)$ satisfying
\begin{equation}
\label{eq:square}
	f(w)- f(\wc)+f(w\tau)-f(\wc\tau) \equiv 0 \mod{(t_{w(d+1)}-t_{w(d)})^2},
\end{equation} 
for any $w \in \Sn$.
The dot actions on $H^*_T(\Gc)$ and $H^*_T(\tG,s)$ are defined in the same way as that on $H^*_T(\G)$,
that is, $(\sigma \cdot f)({}^\circ w) = \sigma(f({}^\circ \sigma^{-1} w))$ for $\sigma \in \Sn$.

Let $x_i$ be an element of $\mathrm{Map}(\Sn \cup \Snc, \C[t_1,\ldots, t_n])$ defined by
\[
	x_i(\wc \tau) = x_i(w) = t_{w(i)}
\]
for $w \in \Sn$.
Then $x_i \in H^*_T(\tG,s)$.
We write $x_i$ also for the restriction of $x_i$ onto $\G$ or $\Gc$.
On $\G$, $x_i$ is the equivariant Euler class of the tautological line bundle $\Set{(V_\bullet, v) | v \in V_i/V_{i-1}}$ over $X(h)$.

We consider the following four maps. We show in Lemma~\ref{lem:2} below that the image of each map is contained in $H^*_T(\tG,s)$. 

\medskip
(1) $\varphi\colon H^*_T(\Gc)\to H^*_T(\tG,s)$ is an $H^*(BT)$-algebra map defined by 
\[
	\varphi(f)(w) := f(\wc\tau),\qquad
	\varphi(f)(\wc) := f(\wc).
\] 

\medskip
(2) $\psi_!\colon H^{*-2}_T(\G)\to H^*_T(\tG,s)$ is an $H^*(BT)$-module map defined by 
\[
	\psi_!(f)(w) := ((x_{d+1}-x_{d})f)(w),\qquad
	\psi_!(f)(\wc) := 0.
\]

\medskip
(3) $\eta\colon H^*_T(\G_+)\to H^*_T(\tG,s)$ is an $H^*(BT)$-algebra map defined by 
\[
	\eta(f)(w) = \eta(f)(\wc) := f(w).
\]

\medskip
(4) $\rho_!\colon H^{*-2}_T(\G_-)\to H^*_T(\tG,s)$ is an $H^*(BT)$-module map defined by 
\[
	\rho_!(f)(w) := ((x_{d}-x_{d_0})f)(w),\qquad
	\rho_!(f)(\wc) := ((x_{d+1}-x_{d_0})f)(w).
\]
One can easily check that all the maps above commute with the dot action.

\begin{lemma}
\label{lem:2}
The images of all the four maps above are contained in $H^*_T(\tG,s)$. 
\end{lemma}

\begin{proof} $E(\tG)$ consists of three classes $E(\G)$, $E(\Gc)$, and $\{\{w,\wc\}\mid w\in \Sn\}$.
We check the congruence relation in \eqref{eq:graph_cohomology} for each class, and the congruence relation \eqref{eq:square}. 

\begin{enumerate}
\item
The congruence relation in \eqref{eq:graph_cohomology} for $E(\G)$ follows from Lemma~\ref{lem:1} and is obvious for $E(\Gc)$ since $f\in H^*_T(\Gc)$.
As for $\{w,\wc\}$, we have 
\[
	\varphi(f)(w)-\varphi(f)(\wc) = f(\wc\tau)-f(\wc) \equiv 0 \mod{(t_{w(d+1)}-t_{w(d)})}
\]
because $f\in H^*_T(\Gc)$ and $\tau=(d+1,d)$. 
The congruence relation \eqref{eq:square} follows from 
\[
	\varphi(f)(w)-\varphi(f)(\wc)+\varphi(f)(w\tau)-\varphi(f)(\wc\tau)
	= f(\wc\tau)-f(\wc)+f(\wc)-f(\wc\tau) = 0.
\]

\item
The congruence relation in \eqref{eq:graph_cohomology} for $E(\G)$ and $\{w,\wc\}$ is obvious and that for $E(\Gc)$ is trivial.
The congruence relation \eqref{eq:square} follows from 
\[
\begin{split}
	\psi_!(f)(w)-\psi_!(f)(\wc)+\psi_!(f)(w\tau)-\psi_!(f)(\wc\tau)
	=(t_{w(d+1)}-t_{w(d)})f(w)+(t_{w\tau(d+1)}-t_{w\tau(d)})f(w\tau)\\
	\equiv (t_{w(d+1)}-t_{w(d)})(f(w)-f(w\tau))
	\equiv 0 \mod{(t_{w(d+1)}-t_{w(d)})^2},
\end{split}
\]
where the last congruence relation holds because $f\in H^*_T(\G)$ and $\tau=(d+1,d)$. 

\item
The congruence relations in \eqref{eq:graph_cohomology} and \eqref{eq:square} are obvious from the definition of $\eta(f)$. 

\item
The congruence relations in \eqref{eq:graph_cohomology} for $E(\G) \setminus \{w,w(d_0,d)\}$ and $E(\Gc) \setminus \{\wc, \wc (d_0,d+1)\}$ are easily checked as in the case of $\psi_!$.
Those for $\{w,w(d_0,d)\}$ and $\{\wc, \wc (d_0,d+1)\}$ are obvious by the definition of $\rho_!$.
As for $\{w,\wc\}$, we have
\[
\begin{split}
	\rho_!(f)(w)-\rho_!(f)(\wc)
	&=(t_{w(d)}-t_{w(d_0)})f(w)-(t_{w(d+1)}-t_{w(d_0)})f(w)\\
	&=(t_{w(d)}-t_{w(d+1)})f(w)
	\equiv 0 \mod{(t_{w(d+1)}-t_{w(d)})}.
\end{split}
\]

The congruence relation \eqref{eq:square} follows from 
\[
\begin{split}
	\rho_!(f)(w)-\rho_!(f)(\wc)+\rho_!(f)(w\tau)-\rho_!(f)(\wc\tau)
	= (t_{w(d)}-t_{w(d+1)})f(w) + (t_{w\tau(d)}-t_{w\tau(d+1)})f(w\tau)\\
	= (t_{w(d)}-t_{w(d+1)})(f(w)-f(w\tau))
	\equiv 0 \mod{(t_{w(d+1)}-t_{w(d)})^2}.
\end{split}
\]
The last congruence relation holds because $f\in H^*_T(\G_-)$ and $\tau=(d+1,d)$. 
\end{enumerate}
\end{proof}

\section{Proof of the modular law for geometric sides}
\label{sect:5}

Under the set up in Section~\ref{sect:4}, we prove the following. 

\begin{theorem}
\label{thm:main}
The homomorphisms 
\[
	\varphi+\psi_!\colon H^*_T(\Gc)\oplus H^{*-2}_T(\G)\to H^*_T(\tG,s),\qquad 
	(\varphi+\psi_!)(f,g) = \varphi(f)+\psi_!(g)
\]
and 
\[
	\eta+\rho_!\colon H^*_T(\G_+)\oplus H^{*-2}_T(\G_-)\to H^*_T(\tG,s),\qquad 
	(\eta+\rho_!)(f_+,f_-) = \eta(f_+)+\rho_!(f_-)
\]
are both isomorphisms as $\Sn$-modules.
\end{theorem}

Since $\Gc$ is isomorphic to $\G$ as labeled graphs by Lemma~\ref{lem:1}, we obtain the following corollary.

\begin{corollary}
There is an isomorphism 
\[
	H^*(\G)\oplus H^{*-2}(\G)\cong H^*(\G_+)\oplus H^{*-2}(\G_-)
\]
as $\Sn$-modules. 
Therefore $H^*(X(h))$ satisfies the modular law.
\end{corollary}

The rest of this section is devoted to the proof of Theorem~\ref{thm:main}. 

\begin{proof}[Proof of the former part in Theorem~\ref{thm:main}] 
Clearly $\varphi(H^*_T(\Gc))\cap \psi_!(H^{*-2}_T(\G))=\{0\}$ and $\varphi, \psi_!$ are injective,
so it suffices to show the surjectivity of $\varphi+\psi_!$. 

Take any element $\tilde{f}\in H^*_T(\tG,s)$ and denote its restriction to $\Gc$ by $f$.
Then $f\in H^*_T(\Gc)$ and $\tilde{f}-\varphi(f)$ vanishes on $V(\Gc)$ by definition of $\varphi$.
Therefore, there is $g\in \mathrm{Map}(V(\G),H^*(BT))$ such that 
\begin{equation}
\label{eq:g}
(\tilde{f}-\varphi(f))(w)=(t_{w(d+1)}-t_{w(d)})g(w) \quad \text{for $w\in V(\G)=\Sn$}.
\end{equation}
We show that $g\in H^{*-2}_T(\G)$.
Then \eqref{eq:g} means that $\tilde{f}-\varphi(f)=\psi_!(g)$, proving the surjectivity of $\varphi+\psi_!$.

We shall check that $g$ satisfies the congruence relation in \eqref{eq:graph_cohomology} for $\G$. 
It follows from \eqref{eq:g} and the definitions of $\varphi$ and $f$ that 
\begin{equation}
\label{eq:gw}
	(t_{w(d+1)}-t_{w(d)})g(w) = (\tilde{f}-\varphi(f))(w)
	= \tilde{f}(w)-f(\wc\tau) = \tilde{f}(w)-\tilde{f}(\wc\tau).
\end{equation}
Let $j<i\leq h(j)$ and set $v = w(i,j)$.  We note that since $x_d$ and $x_{d+1} \in H^*_T(\G)$ and $v=w(i,j)$, we have
\begin{equation}
\label{eq:vdwd}
	t_{v(d+1)} - t_{v(d)} = x_{d+1}(v) - x_{d}(v)
	\equiv x_{d+1}(w) - x_{d}(w) = t_{w(d+1)}-t_{w(d)}
	\mod (t_{w(i)} - t_{w(j)}). 
\end{equation}
We also note that 
\begin{equation} \label{eq:tildefwv1}
	\tilde{f}(w)-\tilde{f}(v) \equiv 0 \equiv \tilde{f}(\wc\tau) -\tilde{f}(\vc\tau)
	\mod  (t_{w(i)} - t_{w(j)})
\end{equation}
since $\tilde{f} \in H^*_T(\tG,s)$ and the labels on $\{w,v\}$ and $\{\wc \tau, \vc \tau\}$ are the same, namely $t_{w(i)} - t_{w(j)}$, by Lemma \ref{lem:1}.
Then, it follows from \eqref{eq:vdwd}, \eqref{eq:gw}, and \eqref{eq:tildefwv1} that 
\[
\begin{split}
	(t_{w(d+1)} - t_{w(d)})(g(w) -g(v)) &\equiv (t_{w(d+1)} - t_{w(d)})g(w) -(t_{v(d+1)}-t_{v(d)})g(v)\\
	&= \tilde{f}(w)-\tilde{f}(\wc\tau) - (\tilde{f}(v)-\tilde{f}(\vc\tau))\\
	&\equiv 0 \mod (t_{w(i)} - t_{w(j)}).
\end{split}
\]
Since $H^*(BT)$ is a polynomial ring, the congruence relation above implies 
\begin{equation} 
\label{eq:g_d+1_d_not}
	g(w) \equiv g(v) \mod  (t_{w(i)} - t_{w(j)})\quad \text{when $(i,j)\not=(d+1,d)$}. 
\end{equation}

When $(i,j)=(d+1,d)$, we have $v = w \tau$. Then, since $(\tau(d+1),\tau(d))=(d,d+1)$, we have  
\begin{align*}
\label{eq:gw-gv2}
	(t_{w(d+1)} - t_{w(d)})(g(w)-g(w\tau))
	&= (t_{w(d+1)} - t_{w(d)})g(w) + (t_{w\tau(d+1)} - t_{w\tau(d)})g(w\tau) \\
	&= \tilde{f}(w) -\tilde{f}(\wc\tau) + \tilde{f}(w\tau)-\tilde{f}(\wc)
	\equiv 0 \mod (t_{w(d+1)}-t_{w(d)})^2,
\end{align*}
where the second identity follows from \eqref{eq:gw} and the last congruence relation follows from  \eqref{eq:square} for $\tilde f \in H^*_T(\tG,s)$.  
Hence we obtain  $$g(w) \equiv g(w\tau)  \mod{(t_{w(d+1)}-t_{w(d)})}.$$
This together with \eqref{eq:g_d+1_d_not} shows that $g\in H^{*-2}_T(\G)$. 
\end{proof}

\begin{proof}[Proof of the latter part in Theorem~\ref{thm:main}] 
It is easy to see $\eta(H^*_T(\G_+))\cap \rho_!(H^{*-2}_T(\G_-))=\{0\}$ and $\eta, \rho_!$ are injective,
so it suffices to prove the surjectivity of $\eta+\rho_!$. 

Take any element $\tilde{f}\in H^*_T(\tG,s)$.
Since the label on the edge $\{w,\wc\}$ is $t_{w(d+1)}-t_{w(d)}$, there is $p\in \mathrm{Map}(\Sn,H^*(BT))$ such that 
\begin{equation}
\label{eq:tilde_f_divisible}
	\tilde{f}(w)-\tilde{f}(\wc) = (t_{w(d+1)}-t_{w(d)})p(w) \quad \text{for $w\in \Sn$}.
\end{equation}
We show that $p\in H^*_T(\G_-)$. The argument is the same as in the former case for $g$ in \eqref{eq:g} being in $H^{*-2}_T(\G)$.  

For a transposition $(i,j)$ with $j<i\le h_-(j)$, we set $v=w(i,j)$ as before.  Then, since $\tilde{f}\in H^*_T(\tG,s)$, we have  
\begin{equation} \label{eq:tildefwv}
\tilde{f}(w)-\tilde{f}(v)\equiv 0\equiv \tilde{f}(\wc)-\tilde{f}(\vc)\mod{(t_{w(i)}-t_{w(j)})}.
\end{equation}
It follows from \eqref{eq:vdwd}, \eqref{eq:tilde_f_divisible} and \eqref{eq:tildefwv} that 
\begin{equation*} 
\begin{split}
	(t_{w(d+1)}-t_{w(d)})(p(w)-p(v))
	&\equiv (t_{w(d+1)}-t_{w(d)})p(w)-(t_{v(d+1)}-t_{v(d)})p(v)\\
	&= \tilde{f}(w)-\tilde{f}(\wc)-(\tilde{f}(v)-\tilde{f}(\vc))\\
	&\equiv 0\mod{(t_{w(i)}-t_{w(j)})}.
\end{split}
\end{equation*}
Therefore, 
\begin{equation} 
\label{eq:p_d_d+1_not}
	p(w)\equiv p(v)\mod{(t_{w(i)}-t_{w(j)})}\quad \text{when $(i,j)\not = (d+1,d)$}.
\end{equation}

When $(i,j)=(d+1,d)$, we have $v=w\tau$ and it follows from \eqref{eq:tilde_f_divisible} and \eqref{eq:square} for $\tilde{f}$ that 
\[
	(t_{w(d+1)}-t_{w(d)})(p(w)-p(w\tau))
	=\tilde{f}(w)-\tilde{f}(\wc)+\tilde{f}(w\tau)-\tilde{f}(\wc\tau)
	\equiv 0 \mod (t_{w(d+1)}-t_{w(d)})^2.
\]
Hence we obtain 
\begin{equation*}
	p(w)\equiv p(w\tau)\mod{(t_{w(d+1)}-t_{w(d)})}.
\end{equation*}
This together with \eqref{eq:p_d_d+1_not} shows that $p\in H^{*-2}_T(\G_-)$.

We shall observe that $\tilde{f}+\rho_!(p)\in \eta(H^*_T(\G_+))$, which implies the surjectivity of $\eta+\rho_!$.  It follows from the definition of $\rho_!$ that 
\[
\begin{split}
	(\tilde{f}+\rho_!(p))(w)&=\tilde{f}(w)+(t_{w(d)}-t_{w(d_0)})p(w),\\
	(\tilde{f}+\rho_!(p))(\wc)&=\tilde{f}(\wc)+(t_{w(d+1)}-t_{w(d_0)})p(w)\\
	&=\tilde{f}(w)-(t_{w(d+1)}-t_{w(d)})p(w)+(t_{w(d+1)}-t_{w(d_0)})p(w)
	\qquad (\text{by \eqref{eq:tilde_f_divisible}})\\
	&=\tilde{f}(w)+(t_{w(d)}-t_{w(d_0)})p(w).
\end{split}
\]
Therefore, 
\begin{equation}
\label{eq:rho!(p)}
	(\tilde{f}+\rho_!(p))(w)=(\tilde{f}+\rho_!(p))(\wc) \quad\text{for any $w\in \Sn$}.
\end{equation} 
Moreover, since $\tilde{f}+\rho_!(p)\in H^*_T(\tG,s)$, $\tilde{f}+\rho_!(p)$ restricted to the subgraphs $\G$ and $\Gc$ satisfies the congruence relation for them and hence for $\G_+$.
This together with \eqref{eq:rho!(p)} shows that $\tilde{f}+\rho_!(f)\in \eta(H^*_T(\G_+))$, proving the surjectivity of $\eta+\rho_!$. 
\end{proof}

\section{The case of twins}
\label{sect:6}

In this section we point out how our argument changes for the proof of the modular law for twins $Y(h)$.

The following remark corresponds to Remark \ref{rem:Xht}
and it is sufficient to consider the modular triple of type (C).
\begin{remark}
	The correspondence $Tg \mapsto TgJ$ gives a diffeomorphism $Y(h) \cong Y(h^t)$.
	In terms of GKM graphs, an edge $\{w, w(i,j)\}$ corresponds to $\{ww_0, w(i,j)w_0\}$.
	This gives an isomorphism between GKM graphs of $Y(h)$ and $Y(h^t)$ with a change of the labels
	$t_i - t_j \mapsto w_0(t_i - t_j) = t_{n+1-i} - t_{n+1-j}$.
	This induces a weakly $H^*(BT)$-algebra isomorphism $H^*_T(Y(h^t)) \to H^*_T(Y(h))$ which commutes with the dagger action \eqref{eq:dagger_action}.
\end{remark}

We define labeled graphs $\Gc_Y$ and $\tG_Y$ as follows:
the underlying graphs are the same as in the case of $X(h)$,
and the label on each edge is changed from $t_{w(i)}-t_{w(j)}$ to $t_i - t_j$.
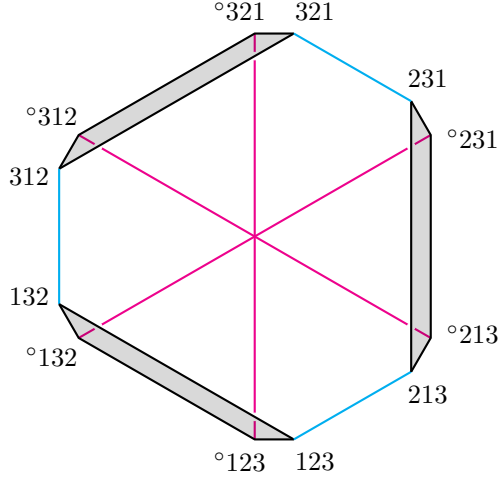
\begin{figure}[H]
\begin{center}
\begin{tikzpicture}
	\draw [thick, magenta] (0,2.7)--(0,-2.7);
	\draw [thick, magenta] (2.34,1.35)--(-2.34,-1.35);
	\draw [thick, magenta] (2.34,-1.35)--(-2.34,1.35);
	\filldraw [white](0,2.4) circle (0.05);
	\filldraw [white](0,-2.4) circle (0.05);
	\filldraw [white](-2.08,1.2) circle (0.05);
	\filldraw [white](-2.08,-1.2) circle (0.05);
	\filldraw [white](2.08,1.2) circle (0.05);
	\filldraw [white](2.08,-1.2) circle (0.05);
	\draw [fill = black, opacity=0.15] (-2.6,0.9)--(-2.34,1.35)--(0,2.7)--(0.52,2.7);
	\draw [thick] (-2.6,0.9)--(-2.34,1.35)--(0,2.7)--(0.52,2.7);
	\draw [thick] (0.52,2.7)--(-2.6,0.9);
	\draw [thick, cyan] (-2.6,0.9)--(-2.6,-0.9);
	\draw [fill = black, opacity=0.15] (-2.6,-0.9)--(-2.34,-1.35)--(0,-2.7)--(0.52,-2.7);
	\draw [thick] (-2.6,-0.9)--(-2.34,-1.35)--(0,-2.7)--(0.52,-2.7);
	\draw [thick] (0.52,-2.7)--(-2.6,-0.9);
	\draw [thick, cyan] (0.52,-2.7)--(2.08,-1.8);
	\draw [fill = black, opacity=0.15] (2.08,1.8)--(2.34,1.35)--(2.34,-1.35)--(2.08,-1.8);
	\draw [thick] (2.08,1.8)--(2.34,1.35)--(2.34,-1.35)--(2.08,-1.8);
	\draw [thick] (2.08,-1.8)--(2.08,1.8);
	\draw [thick, cyan] (2.08,1.8)--(0.52,2.7);
	\draw(0.8,-3) node{$123$};
	\draw(-0.2,-3) node{${}^\circ123$};
	\draw(0.8,3) node{$321$};
	\draw(-0.2,3) node{${}^\circ321$};
	\draw(2.3,-2.1) node{$213$};
	\draw(2.9,-1.3) node{${}^\circ213$};
	\draw(2.3,2.1) node{$231$};
	\draw(2.9,1.3) node{${}^\circ231$};
	\draw(-3,-0.8) node{$132$};
	\draw(-2.7,-1.6) node{${}^\circ132$};
	\draw(-3,0.8) node{$312$};
	\draw(-2.7,1.6) node{${}^\circ312$};
\end{tikzpicture}
\caption{Labeled graph $\tG_Y$ with emphasized $4$-gons when $h = (2,3,3)$ and $d = 2$.}
\label{fig:tGY}
\end{center}
\end{figure}

\begin{lemma}
The map $\Phi \colon w \mapsto \wc \tau$ gives an isomorphism from $\G_Y$ to $\Gc_Y$ through the automorphism $\tau$ that exchanges $t_d$ and $t_{d+1}$.
In particular, $H^*_T(\Gc_Y) \cong H^*_T(\G_Y)$ as $\Sn$-modules by $f \mapsto \tau(f\circ \Phi)$.
\end{lemma}

In the case of $Y(h)$, $s$ changes to
\[
	s_Y(w) = (-1)^{l(w)}, \quad s_Y({}^\circ w) =  (-1)^{l(w)+1} \quad \text{for any } w \in \Sn,
\]
where $l$ is the length function.
Then we have
\[
	H^*_T(\tG_Y,s_Y) = \Set{f \in H^*_T(\tG_Y) | \sum_{v \in V(K_w)} s_Y(v)f(v) \equiv 0 \bmod{(t_{d} - t_{d+1})^2} \text{ for any } w \in \Sn}.
\]
In other words, $H^*_T(\tG_Y, s_Y)$ consists of all $f \in H^*_T(\tG_Y)$ satisfying
\[
	f(w) - f(\wc) +f(\wc\tau) - f(w\tau) \equiv 0 \mod{(t_{d+1}-t_{d})^2}
\]
for any $w \in \Sn$.

In Section \ref{sect:4}, we defined $4$ maps $\varphi, \psi_!, \eta,\rho_!$.
The necessary changes for them are replacing $x$ with $t$ and redefining $\varphi(f)(w):=\tau(f(\wc\tau))$.
This change arises from the change of labels.
Note that $\varphi$ is a weakly $H^*(BT)$-algebra map in this case.
Then all the maps commute with the dagger action. 
The dagger actions on $H^*_T(\Gc_Y)$ and $H^*_T(\tG_Y,s_Y)$ are defined in the same way as that on $H^*_T(\G_Y)$,
that is, $(\sigma \dagger f)({}^\circ w) = f({}^\circ \sigma^{-1} w)$ for $\sigma \in \Sn$.

\appendix

\section{Geometrical counterpart}
\label{app:tX}

We shall show a geometrical object corresponding to $\tG$ for geometrical understanding. 
It is $\tX(h)$ defined as follows.
When $V_\bullet \in X(h_+)$, the dimension of $SV_{d_0}/(V_{d-1} \cap SV_{d_0})$ is 1 or 0
since $SV_{d_0-1} \subset V_{h(d_0-1)} \subset V_{d-1}$.
In particular, $\dim SV_{d_0}/(V_{d-1} \cap SV_{d_0}) = 0$ implies
$SV_{d_0} \subset V_{d-1}$ and then $V_\bullet \in X(h_-)$.
Hence
\[
	\tX(h) = \Set{(V_\bullet, l) | V_\bullet \in X(h_+),\, l \in P(V_{d+1}/V_{d-1}), \, SV_{d_0} \subset l + V_{d-1}}
\]
is the blow-up of $X(h_+)$ along $X(h_-)$.
Moreover,  for $(V_\bullet, l)$,  the correspondence
\begin{equation}
\label{eq:fibration}
	(V_\bullet, l) \mapsto (V_0 \subset V_1 \subset \cdots \subset V_{d-1} \subset l + V_{d-1} \subset V_{d+1} \subset \cdots \subset V_{n}),
\end{equation}
that is, replacing $V_d$ by $l + V_{d-1}$ gives a map $\tX(h) \to X(h)$ and it is a fiber bundle with fiber $\C P^1$ since $h(d) = h(d+1)$.
This fiber $\C P^1$ means the direction of a line $V_d/V_{d-1}$ in a plane $V_{d+1}/V_{d-1}$.
The torus $T$ acts on $\tX(h)$ naturally and the fixed point set is
\[
	\tX(h)^T = \Set{(V_\bullet(w),l) | 
	w \in \Sn,\, l = \langle \e_{w(d)}\rangle \text{ or } \langle \e_{w(d+1)}\rangle},
\]
where $V_\bullet(w)$ denotes the permutation flag associated with $w$.
We denote the fixed point $(V_\bullet(w), \langle \e_{w(d)}\rangle)$ as $w$ and $(V_\bullet(w), \langle \e_{w(d+1)}\rangle)$ as ${}^\circ w$ in $\tG$.
The $4$-gon $K_w$ with vertex set $\{w, {}^\circ w, w\tau, {}^\circ w \tau \}$ in $\tG$
corresponds to $\C P^1 \times \C P^1$ in $\tX(h)$
which is the inverse image of $\{V_\bullet \mid V_k = V_k(w) \text{ for } k \neq d\} = \C P^1 \subset X(h)$ under $\eqref{eq:fibration}$.
\begin{figure}[H]
\begin{center}
\begin{tikzpicture}
	\draw [thick] (0,0)--(0,2.5)--(2.5,2.5)--(2.5,0)--cycle;
	\draw [thick, magenta] (0,0)--(2.5,2.5);
	\draw [magenta] (1.6,0.45) node {$l = \langle \e_{w(d)}\rangle$};
	\draw (0,0)[fill = black] circle (3pt);
	\draw (0,2.5)[fill = black] circle (3pt);
	\draw (2.5,2.5)[fill = black] circle (3pt);
	\draw (2.5,0)[fill = black] circle (3pt);
	\draw(-1.8,0) node{$(V_\bullet(w),\langle \e_{w(d)}\rangle) \leftrightarrow w$};
	\draw(-2.1,2.5) node{$(V_\bullet(w),\langle \e_{w(d+1)}\rangle) \leftrightarrow \wc$};
	\draw(4.7,0) node{$w \tau \leftrightarrow (V_\bullet(w\tau),\langle \e_{w\tau(d)}\rangle)$};
	\draw(4.9,2.5) node{$\wc \tau \leftrightarrow (V_\bullet(w\tau),\langle \e_{w\tau(d+1)}\rangle)$};
	\draw (5.5,2.2)--(6.7,2.2);
	\draw (6.2,1.9) node {$=\e_{w(d)}$};
\end{tikzpicture}
\caption{$K_w$ and the geometrical counterpart.}
\label{fig:square}
\end{center}
\end{figure}
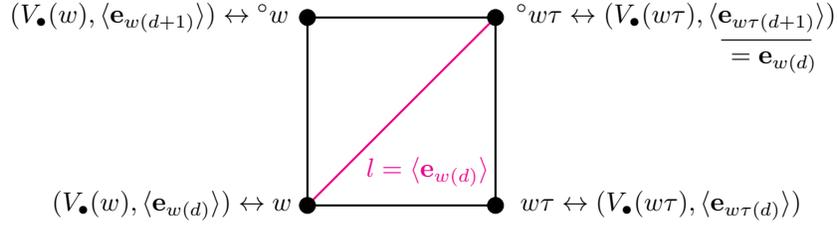
See Figure \ref{fig:square}, and note that
\begin{itemize}
	\item
	taking $V_\bullet = V_\bullet(w)$ or $V_\bullet(w\tau)$ leads us to $\{0,\infty\} \times \C P^1$ which connects $w$ with ${}^\circ w$ and $w\tau$ with ${}^\circ w\tau$,
	\item
	taking $l = \langle \e_{w(d)}\rangle$ or $\langle \e_{w(d+1)}\rangle$ leads us to $\C P^1 \times \{0,\infty\}$ which connects $w$ with ${}^\circ w\tau$ and $w\tau$ with ${}^\circ w$,
	\item
	taking $l = V_d/V_{d-1}$ or $(V_d/V_{d-1})^\perp$ leads us to the ``diagonal'' $\C P^1$'s which connect $w$ with $w \tau$ and ${}^\circ w$ with ${}^\circ w \tau$.
\end{itemize}
This geometrical situation warrants the signs of vertices in \eqref{eq:square}.

\begin{remark}
The labeled graph $\tG_Y$ corresponds to the blow-up $\tY(h)$ of $Y(h_+)$ along $Y(h_-)$,
where Kiem and Lee write $Y(h)$ as $Y_h$ and $\tY(h)$ as $\tY_{\mathbf{h}}$ in \cite{ki-le23}.
See \cite[(4.10) and (4.12)]{ki-le23} for $\C P^1 \times \C P^1$ in $\tY(h)$ to warrant $s_Y$
in Section \ref{sect:6}.
\end{remark}

\begin{remark}	
At the end of this appendix, we mention an interesting phenomenon.
Recall that we have a fiber bundle
\[
	\C P^1 \to \tX(h) \to X(h).
\]
On the other hand, according to Kiem and Lee \cite{ki-le23},
there is a ``reverse'' fiber bundle
\[
	Y(h) \to \tY(h) \to \C P^1.
\]

These fibrations can be described as homomorphisms of the corresponding labeled graphs.
A graph homomorphism may collapse an edge to a vertex.
Note that, for $f \in H^*_T(\tG)$, we have
\[
	f(w) - f(\wc \tau) = f(w) - f(w\tau) + f(w \tau) - f(\wc \tau) \equiv 0 \mod(t_{w(d+1)} - t_{w(d)}).
\]
In this remark, we redefine $\tG$ so that it has additional edges $\{\{w, \wc \tau\} \mid w \in \Sn\}$ with label $t_{w(d+1)} - t_{w(d)}$.
This modification does not change its (equivariant) cohomology ring.
The labeled graph homomorphism corresponding to $\tX(h) \to X(h)$ defined by $\eqref{eq:fibration}$ is given by
\[
	p \colon \tG \to \G, \quad p(w) = p(\wc \tau) = w
\]
with a fiber GKM graph which is the induced labeled subgraph on $\{w, \wc \tau\}$
(see Figure $\ref{fig:square}$).
The fiber is the GKM graph of $\C P^1 = P(\langle \e_{w(d)} , \e_{w(d+1)}\rangle)$.
On the other hand,
let $\Gamma_\tau$ be the GKM graph of $\C P^1 = P(\langle \e_{d} , \e_{d+1}\rangle)$ with vertices $0$ and $\infty$.
Then we have
\[
	q \colon \tG_Y \to \Gamma_\tau, \quad q(w) = 0, \ q(\wc) = \infty
\]
with a fiber GKM graph $\G_Y \cong \Gc_Y$.
\end{remark}

\section{Unicellular LLT polynomials}

It is known that unicellular LLT polynomials satisfy the modular law.
For example, see \cite[Proposition 18]{alex21}.
We give an elementary proof for readers' convenience.

\begin{theorem}
\label{thm:ml_LLT}
	Unicellular LLT polynomials satisfy the modular law.
\end{theorem}

\begin{proof}
	Let $1\leq d_0<d<n$ and $h$ be a Hessenberg function which satisfies the condition $\mathrm{(C)}$, that is,
	$h(d)=h(d+1)$ and $h^{-1}(d)=\{d_0\}$.
	Then $h_-$ and $h_+$ are as follows
	\[
		h_-(j)=
		\begin{cases}
			d-1 & (j = d_0)\\
			h(j) & (j \neq d_0) 
		\end{cases}
		\quad \text{and} \quad
		h_+(j)=
		\begin{cases}
			d+1 & (j = d_0)\\
			h(j) & (j \neq d_0)
		\end{cases}.
	\]
	For this modular triple $(h_-,h,h_+)$, what we have to show is
	\begin{equation}
	\label{eq:ml_LLT}
		\LLT[h_+] - \LLT = q(\LLT - \LLT[h_-] ).
	\end{equation}
	
	Let $C = C(G_{h_-})$ be the set of all colorings of $G_{h_-}$, that is, all maps $[n] \to \N$.
	Then $C$ is decomposed into the following four subsets:
	\begin{itemize}
		\item
		$C_{<<} = \Set{\kappa \in C | \kappa(d_0) < \kappa(d), \, \kappa(d_0) < \kappa(d+1)}$
		\item
		$C_{<\geq} = \Set{\kappa \in C | \kappa(d_0) < \kappa(d), \, \kappa(d_0) \geq \kappa(d+1)}$
		\item
		$C_{\geq<} = \Set{\kappa \in C | \kappa(d_0) \geq \kappa(d), \, \kappa(d_0) < \kappa(d+1)}$
		\item
		$C_{\geq \geq} = \Set{\kappa \in C | \kappa(d_0) \geq \kappa(d), \, \kappa(d_0) \geq \kappa(d+1)}$.
	\end{itemize}
	Let
	\[
		\Asc_-(\kappa) = \Set{\{i,j\}\in E(G_{h_-}) | j<i,\ \kappa(j)<\kappa(i)}.
	\]
	For a coloring $\kappa$ of $G_{h_-}$, we write $\asc_-(\kappa)$ instead of $\asc(\kappa)$ in Appendices to avoid confusions.
	By definition
	\[
		\asc_-(\kappa) = |\Asc_-(\kappa)|.
	\]
	Note that $\{d+1,d\}$ is an edge of $G_{h_-}$ and then
	\begin{itemize}
		\item
		$\kappa \in C_{<\geq} \Rightarrow \kappa(d) > \kappa(d+1) \Leftrightarrow \{d+1,d\} \not \in \Asc_-(\kappa)$ and
		\item
		$\kappa \in C_{\geq<} \Rightarrow \kappa(d) < \kappa(d+1)\Leftrightarrow \{d+1,d\} \in \Asc_-(\kappa)$.
	\end{itemize}
	Let $\tau = (d+1,d) \in \Sn$.
	For a coloring $\kappa$ of $G_{h_-}$,
	the composition $\kappa \circ \tau $ is also a coloring of $G_{h_-}$.
	Then the composition with $\tau$ gives a bijection $C_{<\geq} \to C_{\geq<}$.
	When $e \neq d, d+1$, we have
	\[
		\{ d, e\} \in \Asc_-(\kappa) \Leftrightarrow \{ d+1, e\} \in \Asc_-(\kappa \circ \tau)
		\quad \text{and} \quad
		\{ d+1, e\} \in \Asc_-(\kappa) \Leftrightarrow \{ d, e\} \in \Asc_-(\kappa \circ \tau).
	\]
	Hence, for $\kappa \in C_{<\geq}$, we have
	\begin{equation}
	\label{eq:asc_LLT}
		\asc_-(\kappa) + 1 = \asc_-(\kappa \circ \tau)
	\end{equation}
	since $\{d+1,d\} \not \in \Asc_-(\kappa)$ and $\{d+1,d\} \in \Asc_-(\kappa \circ \tau)$.
	Since the subscripts of the decomposition of $C$ corresponds to the ascents of the edges $\{d,d_0\}$ and $\{d+1,d_0\}$ which are not contained in $G_{h_-}$,
	we have
	\begin{align*}
		\LLT[h_+]
		&= q^2 \sum_{\kappa \in C_{<<}} z_\kappa q^{\asc_-(\kappa)}
		+ q\sum_{\kappa \in C_{<\geq}} z_\kappa q^{\asc_-(\kappa)}
		+ q\sum_{\kappa \in C_{\geq<}} z_\kappa q^{\asc_-(\kappa)}
		+ \sum_{\kappa \in C_{\geq\geq}} z_\kappa q^{\asc_-(\kappa)},\\
		\LLT 
		&= \phantom{{}^2}q\sum_{\kappa \in C_{<<}} z_\kappa q^{\asc_-(\kappa)}
		+ q\sum_{\kappa \in C_{<\geq}} z_\kappa q^{\asc_-(\kappa)}
		+ \phantom{q}\sum_{\kappa \in C_{\geq<}} z_\kappa q^{\asc_-(\kappa)}
		+ \sum_{\kappa \in C_{\geq\geq}} z_\kappa q^{\asc_-(\kappa)},\\
		\LLT[h_-]
		&= \phantom{q^2}\sum_{\kappa \in C_{<<}} z_\kappa q^{\asc_-(\kappa)}
		+ \phantom{q}\sum_{\kappa \in C_{<\geq}} z_\kappa q^{\asc_-(\kappa)}
		+ \phantom{q}\sum_{\kappa \in C_{\geq<}} z_\kappa q^{\asc_-(\kappa)}
		+ \sum_{\kappa \in C_{\geq\geq}} z_\kappa q^{\asc_-(\kappa)}.
	\end{align*}
	Now we compute each side of \eqref{eq:ml_LLT}:
	\begin{align*}
		\LLT[h_+] - \LLT &= ( q^2- q) \sum_{\kappa \in C_{<<}} z_\kappa q^{\asc_-(\kappa)} + ( q- 1) \sum_{\kappa \in C_{\geq<}} z_\kappa q^{\asc_-(\kappa)},\\
		\LLT - \LLT[h_-] &= \phantom{{}^2} ( q- 1) \sum_{\kappa \in C_{<<}} z_\kappa q^{\asc_-(\kappa)} + ( q- 1) \sum_{\kappa \in C_{<\geq}} z_\kappa q^{\asc_-(\kappa)}.
	\end{align*}
	By \eqref{eq:asc_LLT} and the bijection $C_{<\geq} \to C_{\geq<}$, we have
	\begin{equation}
	\label{eq:color_sum}
		\sum_{\kappa \in C_{<\geq}} z_\kappa q^{\asc_-(\kappa)}
		= \sum_{\kappa \in C_{<\geq}} z_\kappa q^{\asc_-(\kappa \circ \tau)-1}
		= \frac1q\sum_{\kappa \in C_{\geq<}} z_\kappa q^{\asc_-(\kappa)}.
	\end{equation}
	This shows \eqref{eq:ml_LLT}.
	
	The same argument shows the modular relation for modular triples of type $\mathrm{(R)}$.
\end{proof}

\begin{remark}
	Our proof is in the same way as Alexandersson's proof for \cite[Proposition 18]{alex21} in essential.
\end{remark}

\section{Chromatic symmetric functions}

We also give an elementary, direct proof of the modular law for $\csf$.
As far as we know, other proofs of the modular law for $\csf$ use the modular law for $\LLT$ and the relation between $\csf$ and $\LLT$ by \cite[Proposition 3.5]{ca-me17}.

\begin{theorem}
	Chromatic symmetric functions satisfy the modular law.
\end{theorem}

\begin{proof}
	Let $d$, $d_0$, $h$, $h_-$, and $h_+$ be the same ones in the the proof of Theorem \ref{thm:ml_LLT}.
	Then what we have to show is
	\begin{equation}
	\label{eq:ml_csf}
		\csf[h_+] - \csf = q(\csf - \csf[h_-] ).
	\end{equation}
	Note that $\{d+1,d\}$ is an edge of $G_{h_-}$.
	Let $C = PC(G_{h_-})$ be the set of all proper colorings of $G_{h_-}$.
	Then $C$ is decomposed into the following nine subsets.
	\begin{itemize}
		\item
		$C_{<<} = \Set{\kappa \in C | \kappa(d_0) < \kappa(d), \, \kappa(d_0) < \kappa(d+1)}$
		\item
		$C_{<=} = \Set{\kappa \in C | \kappa(d_0) < \kappa(d), \, \kappa(d_0) = \kappa(d+1)}, \qquad \kappa \in C_{<=} \Rightarrow \kappa(d) > \kappa(d+1)$
		\item
		$C_{<>} = \Set{\kappa \in C | \kappa(d_0) < \kappa(d), \, \kappa(d_0) > \kappa(d+1)}, \qquad \kappa \in C_{<>} \Rightarrow  \kappa(d) > \kappa(d+1)$
		\item
		$C_{=<} = \Set{\kappa \in C | \kappa(d_0) = \kappa(d), \, \kappa(d_0) < \kappa(d+1)}, \qquad \kappa \in C_{=<} \Rightarrow \kappa(d) < \kappa(d+1)$
		\item
		$C_{==} = \Set{\kappa \in C | \kappa(d_0) = \kappa(d), \, \kappa(d_0) = \kappa(d+1)} = \emptyset$
		\item
		$C_{=>} = \Set{\kappa \in C | \kappa(d_0) = \kappa(d), \, \kappa(d_0) > \kappa(d+1)}, \qquad \kappa \in C_{=>} \Rightarrow \kappa(d) > \kappa(d+1)$
		\item
		$C_{><} = \Set{\kappa \in C | \kappa(d_0) > \kappa(d), \, \kappa(d_0) < \kappa(d+1)}, \qquad \kappa \in C_{><} \Rightarrow \kappa(d) < \kappa(d+1)$
		\item
		$C_{>=} = \Set{\kappa \in C | \kappa(d_0) > \kappa(d), \, \kappa(d_0) = \kappa(d+1)}, \qquad \kappa \in C_{>=} \Rightarrow \kappa(d) < \kappa(d+1)$
		\item
		$C_{>>} = \Set{\kappa \in C | \kappa(d_0) > \kappa(d), \, \kappa(d_0) > \kappa(d+1)}$
	\end{itemize}
	Some proper colorings of $G_{h_-}$ are naturally considered as proper colorings of $G_h$ or $G_{h_+}$,
	that is,
	\begin{align*}
		\text{the set of all proper colorings of $G_h$ is }&
		C_{<<} \sqcup C_{<=} \sqcup C_{<>} \sqcup C_{><} \sqcup C_{>=} \sqcup C_{>>}\\
		\text{and }\text{the set of all proper colorings of $G_{h_+}$ is }&
		C_{<<} \sqcup  \phantom{,\sqcup C_{<=}} C_{<>} \sqcup C_{><} \sqcup\phantom{,\sqcup C_{<=}} C_{>>}.
	\end{align*}
	Then we have
	\begin{align}
	\begin{split}
	\label{eq:diff_csf+}
	&\csf[h_+] - \csf\\
		&= (q^2-q)\sum_{\kappa \in C_{<<}} z_\kappa q^{\asc_-(\kappa)}
		- q \sum_{\kappa \in C_{<=}} z_\kappa q^{\asc_-(\kappa)}
		+ (q-1)\sum_{\kappa \in C_{><}} z_\kappa q^{\asc_-(\kappa)}
		- \sum_{\kappa \in C_{>=}} z_\kappa q^{\asc_-(\kappa)}
	\end{split}\\
	\begin{split}
	\label{eq:diff_csf}
		&\csf - \csf[h_-]\\
		&= (q-1)\sum_{\kappa \in C_{<<}} z_\kappa q^{\asc_-(\kappa)}
		+ (q-1)\sum_{\kappa \in C_{<=}} z_\kappa q^{\asc_-(\kappa)}
		+ (q-1)\sum_{\kappa \in C_{<>}} z_\kappa q^{\asc_-(\kappa)}\\
		&\quad - \sum_{\kappa \in C_{=<}} z_\kappa q^{\asc_-(\kappa)}
		- \sum_{\kappa \in C_{=>}} z_\kappa q^{\asc_-(\kappa)}
	\end{split}
	\end{align}
	By a similar argument to that for \eqref{eq:color_sum} and the bijection $C_{<=}  \to C_{=<}$ given by the composition with $\tau$, we have
	\[
		q\sum_{\kappa \in C_{<=}} z_\kappa q^{\asc_-(\kappa)}
		- \sum_{\kappa \in C_{=<}} z_\kappa q^{\asc_-(\kappa)}
		= 0.
	\]
	Hence these terms in $\eqref{eq:diff_csf}$ are cancelled and we obtain
	\begin{align}
	\begin{split}
	\label{eq:diff_csf2}
		&\csf - \csf[h_-]\\
		&= (q-1)\sum_{\kappa \in C_{<<}} z_\kappa q^{\asc_-(\kappa)}
		-\sum_{\kappa \in C_{<=}} z_\kappa q^{\asc_-(\kappa)}
		+ (q-1)\sum_{\kappa \in C_{<>}} z_\kappa q^{\asc_-(\kappa)}
		- \sum_{\kappa \in C_{=>}} z_\kappa q^{\asc_-(\kappa)}.
	\end{split}
	\end{align}
	Then we obtain \eqref{eq:ml_csf} from \eqref{eq:diff_csf+} and \eqref{eq:diff_csf2} by the two bijections $C_{<>} \to C_{><}$ and $C_{=>}  \to C_{>=}$ given by the composition with $\tau$.
	
	The same argument shows the modular relation for modular triples of type $\mathrm{(R)}$.
\end{proof}

\bigskip
\noindent \textbf{Acknowledgements.} 
This work was partly supported by MEXT Promotion of Distinctive Joint Research Center Program JPMXP0723833165.
The first author is supported in part by JSPS Grant-in-Aid for Young Scientists: 23K12981.
The second author is supported in part by JSPS Grant-in-Aid for Scientific Research 22K03292 and the HSE University Basic Research Program.


\begin{thebibliography}{9}

\bibitem{ab-ni22}
A. Abreu and A. Nigro,
\emph{An update on Haiman's conjectures}, 
arXiv:2206.00073.

\bibitem{ab-ni21}
A. Abreu and A. Nigro. 
\emph{Chromatic symmetric functions from the modular law}. 
J. Combin. Theory Ser. A, 180:Paper No. 105407, 30, 2021

\bibitem{ab-ni21-2}
A. Abreu and A. Nigro. 
\emph{A symmetric function of increasing forests}. 
Forum Math. Sigma 9 (2021), https://doi.org/10.1017/fms.2021.33. 


\bibitem{alex21}
P. Alexandersson,
\emph{LLT polynomials, elementary symmetric functions and melting lollipops},
J. of Alg. Comb. 53 (2021), 299--325. 
%
%

\bibitem{ay-bu21}
A. Ayzenberg and V. Buchstaber,
\emph{Manifolds of isospectral matrices and Hessenberg varieties}, Int. Math. Res. Not. 
IMRN 2021, no. 21, 16671--16692. 

\bibitem{br-ch18}
P. Brosnan and T. Chow,
\emph{Unit interval orders and the dot action on the cohomology of regular semisimple Hessenberg varieties}, Adv. Math. 329 (2018), 955--1001.

\bibitem{ca-me17}
E. Carlsson and A. Mellit,
\emph{A proof of the shuffle conjecture},
J. of Amer. Math. Soc. 31 (2017):661--697.


\bibitem{ma-pr-sh92}
F. De Mari, C. Procesi, and M. A. Shayman, 
\emph{Hessenberg varieties}, 
Trans. Amer. Math. Soc. {332} (1992), no. 2, 529--534. 

\bibitem{go-ko-ma98}
M. Goresky, R. Kottwitz, and R. MacPherson, 
\emph{Equivariant cohomology, Koszul duality, and the localization theorem}, 
Invent. Math. 131 (1998), 25--83.

\bibitem{guay16}
M. Guay-Paquet,
\emph{A second proof of the Shareshian-Wachs conjecture, by way of a new Hopf algebra}, arXiv:1601.05498.

\bibitem{gu-za00}
V. Guillemin and C. Zara,
\emph{Equivariant de Rham theory and graphs},
Sir Michael Atiyah: a great mathematician of the twentieth century,
Asian J. Math. 3 (1999), no.1, 49--76.

\bibitem{gu-za01}
V. Guillemin. C. Zara,
\emph{1-skeleta, Betti numbers, and equivariant cohomology},
Duke Math. J. 107 (2001), 283--349.

%


\bibitem{ki-le22}
Y.H. Kiem and D. Lee,
\emph{Birational geometry of generalized Hessenberg varieties and the generalized Shareshian-Wachs conjecture},
arXiv:2208.12282.

\bibitem{ki-le23}
Y.H. Kiem and D. Lee,
\emph{Geometry of the twin manifolds of regular semisimple Hessenberg varieties and unicellular LLT polynomials},
arXiv:2307.01130.


\bibitem{LLT97}
A. Lascoux, B. Leclerc, and J-Y. Thibon, 
\emph{Ribbon tableaux, Hall--Littlewood functions, quantum affine algebras and unipotent varieties}, 
J. Math. Phys, 38 (1997), 1041--1068.

\bibitem{ma-sa23}
M. Masuda and T. Sato, 
\emph{Unicellular LLT polynomials and twins of regular semisimple Hessenberg varieties}, 
Int. Math. Res. Not. IMRN, rnac359 (2023).

\bibitem{pr-so22}
M. Precup and E. Sommers, 
\emph{Perverse sheaves, nilpotent Hessenberg varieties, and the modular law},
Pure Appl. Math. Q. (to appear),
arXiv:2201.13346.

\bibitem{sh-wa16}
J. Shareshian and M. L. Wachs, 
\emph{Chromatic quasisymmetric functions}, 
Adv. Math. 295 (2016), 497--551. 

\bibitem{tymo08}
J. Tymoczko, 
\emph{Permutation actions on equivariant cohomology of flag varieties}, 
Contemp. Math. 460 (2008), 365--384.

\end{thebibliography}
\end{document}